\documentclass[letterpaper, 10 pt, conference]{ieeeconf}
\IEEEoverridecommandlockouts
\usepackage{cite}
\usepackage{amsmath,amssymb,amsfonts}
\usepackage{graphicx}
\usepackage{textcomp}
\usepackage{xcolor}
\usepackage{mathrsfs}
\usepackage{bm}
\usepackage{algorithm}
\usepackage{algpseudocode}
\usepackage{caption}
\usepackage{subcaption}
\usepackage{stfloats}
\usepackage{wrapfig}
\usepackage{booktabs}
\DeclareMathOperator*{\argmin}{arg\,min}

\newtheorem{Theorem}{Theorem}
\newtheorem{Lemma}{Lemma}

\title{\LARGE \bf
An Occupation-Measure and Frank--Wolfe Framework for Heterogeneous Mean-Field Control
}

\author{Di Yu, Sixiong You and Chaoying Pei%
    \thanks{Di Yu is with the Department of Mathematics and Statistics, Missouri University of Science and Technology, Rolla, MO 65401, USA. Email: {\tt\small dy9dq@mst.edu}}
    \thanks{Sixiong You is with Eli Lilly and Company, Indianapolis, IN 46225, USA. Email: {\tt\small yousixiong@gmail.com}}
    \thanks{Chaoying Pei is with the Department of Mechanical and Aerospace Engineering, Missouri University of Science and Technology, Rolla, MO 65401, USA. Email: {\tt\small cpk4t@mst.edu}}
}

\begin{document}

\maketitle

\begin{abstract}
Heterogeneous mean-field control (MFC) involves multiple interacting populations with distinct dynamics, control constraints, and coupling structures. We develop a heterogeneous occupation-measure mean-field control (OM-MFC) framework that reformulates the population-distribution control problem as an optimization problem over occupation measures. In this representation, potentially nonlinear dynamics enter through linear weak Liouville constraints, while nonlocal inter-population interactions become structured bilinear functionals of the occupation-measure marginals. Exploiting this structure, we derive a Frank--Wolfe (FW) method and show that its linear minimization oracle decomposes into independent population-wise optimal control problems once the current population distributions are fixed, enabling parallel computation without an a priori discretization of the measure space. We further establish a sufficient convexity condition based on the symmetrized matrix-valued interaction kernel, under which Frank--Wolfe with an exact oracle admits the standard \(O(1/k)\) convergence rate. Numerical examples on heterogeneous UAV coordination and three-dimensional search-and-rescue illustrate the framework both within the certified convex regime and for directional interactions beyond that regime.
\end{abstract}

\section{Introduction}

Mean-field control (MFC) has become an important framework for large-scale coordination problems~\cite{albi2017mean}. By replacing agent-wise descriptions with the evolution of population distributions, it provides scalable control formulations whose complexity depends on the state of a representative agent rather than on the number of agents~\cite{zhang2016optimal}. This viewpoint is particularly appealing in applications such as multi-UAV systems~\cite{chen2020mean} and satellite constellations~\cite{yu2026occupation}. Much of the existing MFC literature focuses on the homogeneous setting, where agents share common dynamics, control constraints, and interaction structures. In this setting, a variety of PDE-based~\cite{fornasier2014mean}, geometric and optimal-transport-based~\cite{emerick2023continuum}, and learning-based formulations~\cite{chen2020mean} have been developed, leading to substantial progress in both analysis and computation.

In many practical settings, however, such systems are inherently heterogeneous: different groups may have distinct dynamics, control capabilities, and interaction patterns. This has motivated growing interest in heterogeneous and multi-population mean-field models. A representative early work is~\cite{bensoussan2018mean}, which formulates multi-population MFC and derives the associated adjoint PDE systems. More recently, graphon-based MFC has been used to model heterogeneous and non-exchangeable interactions~\cite{cao2025probabilistic}. Mean-field approximations have also been used to study heterogeneous cooperative multi-agent decision-making, together with approximation guarantees and learning-based algorithms~\cite{mondal2022approximation}. These approaches provide important analytical and modeling tools, but they do not directly yield the same optimization structure exploited here: a measure-space formulation with linear dynamical constraints together with a Frank--Wolfe oracle reducible to classical optimal control subproblems.

In our previous work, homogeneous MFC was formulated as a population-level optimization problem over occupation measures~\cite{yuyoupei2026ACC,yu2026occupation}, building on the classical occupation-measure lifting of optimal control~\cite{lasserre2008nonlinear}. That formulation provides linear dynamical constraints in measure space and supports projection-free Frank--Wolfe optimization without requiring an a priori discretization of the measure space. The present work asks whether this structure survives when multiple heterogeneous populations are coupled through population-dependent interaction costs. The occupation-measure lifting and the Frank--Wolfe methodology are inherited from that work; the heterogeneous coupling is what requires new analysis.



The main contributions are threefold. First, we formulate heterogeneous multi-population MFC as an optimization problem over population-specific occupation measures with linear weak Liouville constraints. Second, we show that the Frank--Wolfe linear minimization oracle remains population-wise separable despite cross-population coupling, reducing each iteration to independent classical optimal control problems. Third, we establish a sufficient condition for convexity through a symmetrized matrix-valued interaction kernel and establish an $O(1/k)$ convergence rate under the corresponding positive-semidefiniteness condition. Numerical examples on heterogeneous UAV coordination and three-dimensional search-and-rescue illustrate the framework in both convex and nonconvex regimes.

\section{Motivation: A Two-Population Finite-Agent Problem}
\label{sec:motivation}
We motivate the heterogeneous OM-MFC formulation through
a finite-agent system with two interacting populations.
The two-population setting simplifies notation; the
formulation extends directly to finitely many populations.

\subsection{Heterogeneous Two-Population Dynamics}

We consider two interacting populations indexed by $a\in\{1,2\}$ over a finite horizon $[0,T]$. For population $a$, let $N_a$ be the number of agents, with state $x_i^a(t)\in\mathcal X\subset\mathbb R^{n_x}$ and control $u_i^a(t)\in\mathcal U_a\subset\mathbb R^{n_{u_a}}$. Their dynamics are
\begin{equation}
\label{eq:dynamic}
\dot x_i^a(t)
=
f_a\bigl(x_i^a(t),u_i^a(t)\bigr),
\qquad i=1,\dots,N_a,
\end{equation}
where $f_a$ is continuous and Lipschitz in the state uniformly in the control. The empirical distribution of population $a$ is
\[
\rho_t^{a,N}
:=
\frac{1}{N_a}
\sum_{i=1}^{N_a}
\delta_{x_i^a(t)}.
\]
Heterogeneity is allowed in $f_a$, $\mathcal U_a$, and $\rho_0^a$.

\subsection{Finite-Agent Heterogeneous Cost Functional}

We define the finite-agent cost
{\footnotesize
\begin{align}
\label{eq:het-N-cost}
J_N
&=
\sum_{a=1}^2
\Bigg[
\frac{1}{N_a}
\sum_{i=1}^{N_a}
\int_0^T
\ell_a\bigl(t,x_i^a(t),u_i^a(t)\bigr)\,dt
+
\int_{\mathcal X}\Psi_a(x)\,d\rho_T^{a,N}(x)
\Bigg]
\notag\\
&+
\sum_{p=1}^2\sum_{q=1}^2
\kappa_{pq}
\int_0^T
\iint
W_{pq}(x-y)
\,d\rho_t^{p,N}(x)
\,d\rho_t^{q,N}(y)\,dt.
\end{align}}
Here $\ell_a$ and $\Psi_a$ are the running and terminal
costs, and $\kappa_{pq}\ge0$ are interaction weights.
The kernels and weights may vary across ordered population
pairs, allowing population-dependent and directional
interactions. The resulting interaction term admits a
symmetrized kernel representation, developed in
Section~\ref{sec:properties}.

\subsection{Occupation-Measure Representation and Motivation}

For each population $a\in\{1,2\}$, let
\(
\Sigma_a=[0,T]\times\mathcal X\times\mathcal U_a,
\)
with $\mathcal M_+(\Sigma_a;T)$ denoting nonnegative Borel measures of total mass $T$ and $\mathcal M_+(\mathcal X;1)$ probability measures on $\mathcal X$.

For an admissible trajectory-control pair
$\omega^a=(x^a(\cdot),u^a(\cdot))$,
define the running and terminal occupation measures by
\begin{align}
\int_{\Sigma_a}\varphi\,d\mu[\omega^a]
&=
\int_0^T
\varphi\bigl(t,x^a(t),u^a(t)\bigr)\,dt,
\label{eq:hetero-running-measure}\\
\int_{\mathcal X}\psi\,d\nu[\omega^a]
&=
\psi\bigl(x^a(T)\bigr),
\label{eq:hetero-terminal-measure}
\end{align}
for all $\varphi\in C(\Sigma_a)$ and $\psi\in C(\mathcal X)$; see~\cite{lasserre2008nonlinear}. Averaging these measures within each population yields empirical occupation-measure pairs $(\mu_1^N,\nu_1^N)$ and $(\mu_2^N,\nu_2^N)$, motivating the population-level formulation introduced next.
\section{Heterogeneous OM-MFC Formulation}
\label{sec:main_problem}

\subsection{Population-Level Relaxation and Problem Formulation}

For each population $a\in\{1,2\}$, let
\[
(\mu_a,\nu_a)
\in
\mathcal M_+(\Sigma_a;T)
\times
\mathcal M_+(\mathcal X;1)
\]
denote population-level running and terminal measures.

Given initial distributions
$\rho_0^1,\rho_0^2\in\mathcal M_+(\mathcal X;1)$,
define
{
\small
\begin{align}
&\Delta_a
:=
\Big\{
(\mu_a,\nu_a):\ 
\int_{\mathcal X}
v(T,x)\,d\nu_a(x)
-
\int_{\mathcal X}
v(0,x)\,d\rho_0^a(x)
\notag\\
&=
\int_{\Sigma_a}
\bigl(
\partial_t v
+
\nabla_xv\cdot f_a(x,u)
\bigr)\,d\mu_a,
\forall v\in C^1([0,T]\times\mathcal X)
\Big\}.
\label{eq:hetero-feasible}
\end{align}}
The overall feasible set is
\(
\Delta:=\Delta_1\times\Delta_2.
\) Throughout, $\mathcal X$ and $\mathcal U_a$ are compact,
$\Delta_a\ne\varnothing$, and $\ell_a$, $\Psi_a$, and
$W_{pq}$ are continuous on their respective domains.

Taking $v(t,x)=h(t)$ in~\eqref{eq:hetero-feasible} gives
$\int_{\Sigma_a}h'\,d\mu_a=\int_0^T h'(t)\,dt$ for every
$h\in C^1([0,T])$, so the time marginal of any feasible $\mu_a$ is
Lebesgue measure on $[0,T]$; the disintegration below is therefore well
defined and $\mu_{a,t}$ is a probability measure for a.e.\ $t$. For each $a$, let
\[
d\mu_a(t,x,u)
=
dt\,\mu_{a,t}(dx,du),
\qquad
\rho_{a,t}:=\pi_x\#\mu_{a,t},
\]
denote the time disintegration of $\mu_a$ and the corresponding state marginal. For $p,q\in\{1,2\}$, define
{\small
\begin{align}
\mathcal I_{pq}(\mu_p,\mu_q)
:={}&
\int_0^T
\iint_{\mathcal X\times\mathcal X}
W_{pq}(x-y)
d\rho_{p,t}(x)
\,d\rho_{q,t}(y)\,dt.
\label{eq:hetero-interaction}
\end{align}}

The heterogeneous OM-MFC objective is
{\small
\begin{align}
J(\mu_1,\nu_1,\mu_2,\nu_2)
={}&
\sum_{a=1}^2
\int_{\Sigma_a}
\ell_a(t,x,u)\,d\mu_a
+
\sum_{a=1}^2
\int_{\mathcal X}
\Psi_a(x)\,d\nu_a
\notag\\
&+
\sum_{p=1}^2
\sum_{q=1}^2
\kappa_{pq}
\mathcal I_{pq}(\mu_p,\mu_q).
\label{eq:hetero-mainObj}
\end{align}}
Accordingly,
\begin{equation}
\label{eq:hetero-OMMFC}
\min_{(\mu_1,\nu_1,\mu_2,\nu_2)\in\Delta}
J(\mu_1,\nu_1,\mu_2,\nu_2).
\tag{P}
\end{equation}

Problem~\eqref{eq:hetero-OMMFC} is an infinite-dimensional optimization problem over coupled population-level measure pairs. The population dynamics enter through linear weak Liouville constraints, whereas the interaction term couples their state marginals.

\subsection{PDE and Relaxed-Control Interpretation}

For any feasible pair
$(\mu_a,\nu_a)\in\Delta_a$,
further disintegration with respect to the state variable gives
\[
\mu_{a,t}(dx,du)
=
\lambda_{a,t,x}(du)\rho_{a,t}(dx),
\]
or equivalently
\[
d\mu_a(t,x,u)
=
dt\,
\lambda_{a,t,x}(du)
\rho_{a,t}(dx),
\]
where
$\lambda_{a,t,x}\in\mathcal M_+(\mathcal U_a;1)$
is a measurable probability kernel. Substitution into the weak Liouville equation gives
\begin{equation}
\partial_t\rho_{a,t}
+
\nabla_x\cdot
\bigl(F_a(t,x)\rho_{a,t}\bigr)
=
0,
\label{eq:hetero-relaxed-pde}
\end{equation}
where
\begin{equation}
F_a(t,x)
=
\int_{\mathcal U_a}
f_a(x,u)\,
d\lambda_{a,t,x}(u).
\end{equation}
The derivation follows the same disintegration and weak continuity-equation argument as in the homogeneous case; see~\cite{yu2026occupation} and~\cite{bogachev2007measure}.

This representation gives a relaxed population-level control interpretation of problem~\eqref{eq:hetero-OMMFC}. In particular, the effective velocity $F_a(t,x)$ belongs to the convex hull of
\(
\{f_a(x,u):u\in\mathcal U_a\},
\)
as in classical relaxed control~\cite{vinter2010optimal}. The probability kernel $\lambda_{a,t,x}$ may therefore be interpreted as a relaxed Markov control law.
When $\lambda_{a,t,x}=\delta_{u_a(t,x)}$, the relaxed dynamics reduce to
those induced by the deterministic feedback law $u=u_a(t,x)$.
Thus, deterministic feedback control is recovered as a special case of the occupation-measure representation. In general, however, an occupation-measure solution need not determine a unique deterministic feedback controller; its natural control representation is the conditional control kernel $\lambda_{a,t,x}$.

Problem~\eqref{eq:hetero-OMMFC} is studied as a
population-level relaxed control problem. Finite-agent
limits and general relaxation-gap analysis are beyond
our scope; classical realization of the FW oracle is
established below under additional assumptions.

\section{Properties}
\label{sec:properties}

\subsection{Convexity Condition}

For $p,q\in\{1,2\}$, define the symmetrized kernels
\begin{equation}
K_{pq}(z):=\tfrac12\bigl[\kappa_{pq}W_{pq}(z)+\kappa_{qp}W_{qp}(-z)\bigr],
\label{eq:sym-kernel}
\end{equation}
and the matrix-valued kernel
$\mathbf K(z):=\bigl(K_{pq}(z)\bigr)_{p,q=1}^{2}$. Whatever the ordered
kernels $W_{pq}$, the symmetrization gives $K_{pq}(z)=K_{qp}(-z)$, i.e. \(\mathbf K(-z)=\mathbf K(z)^{\top}\), which makes the induced quadratic form real and symmetric while leaving
each $\mathbf K(z)$ free to be nonsymmetric; kernels with a non-even
directional factor are therefore covered. A change of
variables gives the exact representation
\begin{align}
&\sum_{p,q=1}^2\kappa_{pq}\mathcal I_{pq}(\mu_p,\mu_q) \nonumber \\
&\quad=\sum_{p,q=1}^2\int_0^T\!\!\iint K_{pq}(x-y)\,
d\rho_{p,t}(x)\,d\rho_{q,t}(y)\,dt,
\label{eq:sym-interaction}
\end{align}
so convexity is governed by $\mathbf K$ rather than by the individual
ordered kernels $\kappa_{pq}W_{pq}$. Following the scalar convention
of~\cite{yu2026occupation}, we call $\mathbf K$ positive
semidefinite if
\begin{equation}
\sum_{p,q=1}^{2}\iint K_{pq}(x-y)\,d\sigma_p(x)\,d\sigma_q(y)\ge0
\label{eq:matrix-psd}
\end{equation}
for every pair $\sigma=(\sigma_1,\sigma_2)$ of finite signed Borel
measures on $\mathcal X$.

\begin{Theorem}[Convexity of the heterogeneous objective]
\label{thm:hetero-conv}
Assume $\mathcal X$ is compact, each $K_{pq}$ is bounded and Borel
measurable on $\mathcal X-\mathcal X$, and $\mathbf K$ satisfies the
positive semidefiniteness condition~\eqref{eq:matrix-psd}. Then the
objective $J$ in~\eqref{eq:hetero-mainObj} is convex on
$\Delta=\Delta_1\times\Delta_2$.
\end{Theorem}

\begin{proof}
By the disintegration established in Section~\ref{sec:main_problem},
$d\mu_a=dt\,\mu_{a,t}$ with $\rho_{a,t}=\pi_x\#\mu_{a,t}$ a probability
measure for a.e.\ $t$; by a.e.\ uniqueness of the disintegration, the map
$\mu_a\mapsto(\rho_{a,t})_t$ is affine on $\Delta_a$.

The running and terminal terms are linear, so it suffices to treat the
interaction term, which by~\eqref{eq:sym-interaction} equals
$\int_0^T Q(\rho_t)\,dt$, where $Q(\sigma):=B(\sigma,\sigma)$ and
\[
B(\sigma,\tilde\sigma)
:=\sum_{p,q=1}^{2}
\iint K_{pq}(x-y)\,d\sigma_p(x)\,d\tilde\sigma_q(y).
\]
Relabelling $p\leftrightarrow q$, exchanging the names of the
integration variables, and $K_{qp}(y-x)=K_{pq}(x-y)$ gives
$B(\tilde\sigma,\sigma)=B(\sigma,\tilde\sigma)$. Hence $B$ is a
symmetric bilinear form, and for
$\rho^\theta=\theta\rho^1+(1-\theta)\rho^0$ and
$\eta_t:=\rho^1_t-\rho^0_t$,
\begin{equation}
Q(\rho_t^\theta)
=\theta Q(\rho^1_t)+(1-\theta)Q(\rho^0_t)
-\theta(1-\theta)\,Q(\eta_t).
\label{eq:quad-identity}
\end{equation}
Each $\eta_{p,t}$ is a difference of probability measures on the compact
set $\mathcal X$, so $\|\eta_{p,t}\|_{\mathrm{TV}}\le2$; since the
$K_{pq}$ are bounded, every integral in $Q(\eta_t)$ is finite and
$\eta_t$ is admissible in~\eqref{eq:matrix-psd}. Therefore
$Q(\eta_t)\ge0$ for a.e.\ $t$, and
integrating~\eqref{eq:quad-identity} over $[0,T]$ gives convexity of the
interaction term. Adding the linear terms, $J$ is convex on $\Delta$.
\end{proof}

Because the weak Liouville constraints are linear, $\Delta$ is convex, so
Theorem~\ref{thm:hetero-conv} yields a convex infinite-dimensional
optimization problem whenever $\mathbf K$ is positive semidefinite. This
is readily checked in the settings of interest: if all $W_{pq}$ coincide
with a scalar positive definite kernel $\phi$---Gaussian and other
translation-invariant kernels with nonnegative Fourier transform
qualify~\cite{wendland2004scattered}---then
$\mathbf K(z)=\phi(z)\,\tfrac12(\kappa+\kappa^{\top})$, and a spectral
decomposition of $\tfrac12(\kappa+\kappa^{\top})$
reduces~\eqref{eq:matrix-psd} to a nonnegative combination of scalar
quadratic forms in $\phi$. Since $\kappa_{pq}\ge0$, it therefore suffices
that
\begin{equation}
\kappa_{11}\kappa_{22}\;\ge\;
\Bigl(\tfrac{\kappa_{12}+\kappa_{21}}{2}\Bigr)^{2},
\label{eq:weight-test}
\end{equation}
i.e., that cross-population coupling not exceed the geometric mean of the
two self-couplings.

\subsection{First Variation and Separable Linearization}

To derive a FW method for problem~\eqref{eq:hetero-OMMFC}, fix
\(
(\mu_1,\nu_1,\mu_2,\nu_2)\in\Delta.
\)
Although convexity is most naturally expressed through the symmetrized kernel $\mathbf K$, it is convenient to compute the first variation directly from the original ordered-pair representation~\eqref{eq:hetero-mainObj}. Both representations define the same scalar objective. Define
{\small
\begin{align}
g_a(t,x,u)
:={}&
\ell_a(t,x,u)
+
\sum_{q=1}^2
\kappa_{aq}
\int_{\mathcal X}
W_{aq}(x-y)
\,d\rho_{q,t}(y)
\notag\\
&+
\sum_{p=1}^2
\kappa_{pa}
\int_{\mathcal X}
W_{pa}(y-x)
\,d\rho_{p,t}(y).
\label{eq:ga-hetero}
\end{align}}

For each $a$, write
\begin{align}
\langle g_a,\delta\mu_a\rangle
&:=
\int_{\Sigma_a} g_a(t,x,u)\,d\delta\mu_a(t,x,u), \\
\langle \Psi_a,\delta\nu_a\rangle
&:=
\int_{\mathcal X}\Psi_a(x)\,d\delta\nu_a(x).
\end{align}

\begin{Lemma}[First variation]
\label{lem:hetero-first-variation}
For two feasible points, let
$(\delta\mu_1,\delta\nu_1,\delta\mu_2,\delta\nu_2)$
denote their difference. The first variation of $J$
at the first point in this direction is
\begin{equation}
\delta J
=
\sum_{a=1}^2
\left[
\langle g_a,\delta\mu_a\rangle
+
\langle\Psi_a,\delta\nu_a\rangle
\right].
\label{eq:hetero-first-variation}
\end{equation}
\end{Lemma}

\begin{proof}
The running and terminal terms are linear. For the interaction term, the variation with respect to $\mu_a$ is
\[
\sum_{q=1}^2
\kappa_{aq}
\mathcal I_{aq}(\delta\mu_a,\mu_q)
+
\sum_{p=1}^2
\kappa_{pa}
\mathcal I_{pa}(\mu_p,\delta\mu_a).
\]
Writing these terms in integral form gives~\eqref{eq:ga-hetero}. Summing over populations and adding the terminal terms yields~\eqref{eq:hetero-first-variation}.
\end{proof}

Equation~\eqref{eq:hetero-first-variation} yields a separable first-order linearization once the current iterate is fixed. The cross-population coupling enters the population-dependent running costs $g_1$ and $g_2$, but these become fixed functions during the linear minimization step. This observation is the basis of the decomposition result developed next.

\section{Frank--Wolfe Method for Heterogeneous Occupation-Measure Mean-Field Control}
\label{sec:fw_methods}

We apply the Frank--Wolfe (FW) method~\cite{jaggi2013revisiting} to
problem~\eqref{eq:hetero-OMMFC}. As in the homogeneous OM-MFC
setting~\cite{yu2026occupation}, FW replaces projection onto the
measure-valued feasible set by a linear minimization oracle.

Write
$z:=(\mu_1,\nu_1,\mu_2,\nu_2)$ and $z_k:=(\mu_1^k,\nu_1^k,\mu_2^k,\nu_2^k).$
The FW iteration is
\begin{align}
z_{k+1}
&=(1-\gamma_k)z_k+\gamma_k\tilde z_k,
\label{eq:hetero-FW-combined}\\
\tilde z_k
&\in
\argmin_{\tilde z\in\Delta}
\sum_{a=1}^2
\left[
\langle g_a^k,\tilde\mu_a\rangle
+
\langle\Psi_a,\tilde\nu_a\rangle
\right],
\label{eq:hetero-LMO}
\end{align}
where $g_a^k$ denotes~\eqref{eq:ga-hetero} evaluated at $z_k$.
Although the original objective couples the populations, the
linearized objective is additive once $z_k$ is fixed.

\begin{Theorem}[Population-wise separability of the FW oracle]
\label{thm:hetero-lmo}
For fixed $z_k\in\Delta$, the linear minimization
problem~\eqref{eq:hetero-LMO} separates into
\begin{equation}
\label{eq:hetero-LMO-a}
\min_{(\tilde\mu_a,\tilde\nu_a)\in\Delta_a}
\left[
\langle g_a^k,\tilde\mu_a\rangle
+
\langle\Psi_a,\tilde\nu_a\rangle
\right],
\qquad a=1,2.
\end{equation}
\end{Theorem}

\begin{proof}
At $z_k$, the functions $g_a^k$ are fixed. Since
$\Delta=\Delta_1\times\Delta_2$ and the objective
in~\eqref{eq:hetero-LMO} is additive across populations, minimization
decomposes into the independent problems~\eqref{eq:hetero-LMO-a}.
\end{proof}

Each population-wise oracle can be realized through classical optimal
trajectories under the following regularity, interiority, and selection assumptions.

\begin{Lemma}[Population-wise realization]
\label{lem:hetero-lmo-realization}
Fix $z_k\in\Delta$. For each $a$, assume $\mathcal U_a$
is compact, $\Psi_a$ is continuous, and $f_a,g_a^k$
are continuous and locally Lipschitz in $x$, uniformly
in $(t,u)$, on a neighborhood of $\mathcal X$.
Assume all admissible relaxed trajectories from
$\rho_0^a$-a.e.\ initial state remain in a compact
$D_a\subset\operatorname{int}\mathcal X$.
Suppose
\begin{equation}
\label{eq:hetero-parametric-OCP}
\inf_{\substack{x(0)=\xi,\ \dot x=f_a(x,u)\\
x(t)\in\mathcal X,\ u(t)\in\mathcal U_a}}
\left\{\int_0^T g_a^k(t,x,u)\,dt+\Psi_a(x(T))\right\}
\end{equation}
admits a measurable selection of classical minimizers
$\omega_{a,k}(\xi)$. Then
\[
(\bar\mu_{a,k},\bar\nu_{a,k})
=\int_{\mathcal X}
(\mu[\omega_{a,k}(\xi)],\nu[\omega_{a,k}(\xi)])
\,d\rho_0^a(\xi)
\]
solves~\eqref{eq:hetero-LMO-a}; combining the populations
solves~\eqref{eq:hetero-LMO}.
\end{Lemma}

\proofsketch
Superposition represents feasible measures by relaxed
trajectories~\cite{bogachev2007measure}. Chattering of the augmented
state--cost dynamics yields classical
approximations~\cite{vinter2010optimal}, which remain admissible since
$\operatorname{dist}(D_a,\partial\mathcal X)>0$. Consequently, classical and relaxed oracle values coincide,
and the measurable classical minimizers attain the
population-wise oracle after aggregation. The argument
follows the proof structure of~\cite[App.~B--C]{yu2026occupation},
with admissibility ensured by the interiority assumption above.
\endproof

\begin{algorithm}[t]
\caption{FW for Heterogeneous OM-MFC}
\label{alg:hetero-FW}
\begin{algorithmic}[1]
\Require $z_0\in\Delta$, step sizes $\{\gamma_k\}_{k\ge0}$.
\For{$k=0,\ldots,K-1$}
    \State Compute $g_a^k$ from~\eqref{eq:ga-hetero}, $a=1,2$.
    \For{$a=1,2$}
        \State Solve~\eqref{eq:hetero-parametric-OCP} for
        $\rho_0^a$-a.e.\ $\xi$ and construct
        $(\bar\mu_{a,k},\bar\nu_{a,k})$.
    \EndFor
    \State
    $(\mu_a^{k+1},\nu_a^{k+1})
    \gets
    (1-\gamma_k)(\mu_a^k,\nu_a^k)
    +
    \gamma_k(\bar\mu_{a,k},\bar\nu_{a,k})$,
    $a=1,2$.
\EndFor
\State \Return $z_K$.
\end{algorithmic}
\end{algorithm}
Under Lemma~\ref{lem:hetero-lmo-realization} and a
classical-trajectory initialization, each FW iterate
is a finite convex combination of aggregated
classical-trajectory measures. For finitely supported
initial distributions, this gives a finite weighted
trajectory ensemble. Continuous initial distributions
may be replaced by fixed empirical distributions for
numerical implementation.
\begin{Theorem}[Convergence rate]
\label{thm:hetero-rate}
Assume the conditions of Theorem~\ref{thm:hetero-conv}
and Lemma~\ref{lem:hetero-lmo-realization} at every iteration.
Let $\{z_k\}$ be generated by Algorithm~\ref{alg:hetero-FW}
with $\gamma_k=2/(k+2)$ and exact linear minimization. Let
\begin{equation*}
C_J:=\!\!\sup_{\substack{z,s\in\Delta\\ \gamma\in(0,1]}}\!\!
\frac{2}{\gamma^{2}}
\Bigl[J\bigl(z+\gamma(s-z)\bigr)-J(z)-\gamma\,\delta J(z;s-z)\Bigr]
\end{equation*}
denote the curvature constant of $J$ over
$\Delta$~\cite{jaggi2013revisiting}, with $\delta J$ as in
Lemma~\ref{lem:hetero-first-variation}. Then
\begin{equation}
C_J\;\le\;32\,T\max_{p,q}\|K_{pq}\|_{\infty}\;<\;\infty,
\label{eq:curvature}
\end{equation}
and for every $k\ge1$,
\begin{equation}
J(z_k)-J^*\;\le\;\frac{2C_J}{k+2},
\qquad J^*:=\inf_{z\in\Delta}J(z).
\label{eq:rate}
\end{equation}
\end{Theorem}

\begin{proof}
For $z,s\in\Delta$ put $d=s-z$.
The marginal map is affine and the interaction term is quadratic, so by
Lemma~\ref{lem:hetero-first-variation}
\[
J(z+\gamma d)-J(z)-\gamma\,\delta J(z;d)
=\gamma^{2}\!\int_0^T\! Q(\rho^{d}_t)\,dt
\]
holds exactly for every $\gamma\in(0,1]$ with $Q$ as in the proof of Theorem~\ref{thm:hetero-conv}. Hence
$C_J=2\sup_{z,s\in\Delta}\int_0^T Q(\rho^{s}_t-\rho^{z}_t)\,dt$. Each
$\rho^{s}_{p,t}-\rho^{z}_{p,t}$ is a difference of probability measures,
so its total variation is at most $2$, giving
$|Q(\rho^{s}_t-\rho^{z}_t)|\le16\max_{p,q}\|K_{pq}\|_{\infty}$ for a.e.\
$t$; integrating over $[0,T]$ yields~\eqref{eq:curvature}. Since $J$ is
convex by Theorem~\ref{thm:hetero-conv} and $C_J<\infty$,
\eqref{eq:rate} follows from the standard Frank--Wolfe descent
argument~\cite[Thm.~1]{jaggi2013revisiting}, which carries over to the
measure setting as in~\cite[App.~D]{yu2026occupation}.
\end{proof}

If the positive semidefiniteness condition fails,
Algorithm~\ref{alg:hetero-FW} remains applicable but~\eqref{eq:rate}
does not, and such cases are reported empirically. The fully-corrective
variant used in Section~\ref{sec:numerical} re-optimizes the weights over the accumulated atoms. With an exact oracle and exact correction, its objective is no larger than that of a FW step from its own current iterate, so the same descent bound applies~\cite{yu2026occupation}.

\section{Numerical Illustration}
\label{sec:numerical}

We illustrate the framework through two numerical studies.
The first examines two-population UAV crossing with Gaussian
kernels under symmetric and heterogeneous weighting, both
satisfying Theorem~\ref{thm:hetero-conv}, and reports the
objective decrease of fully-corrective FW. The second
considers three-dimensional search-and-rescue with directional
cross-population kernels, demonstrating the use of the
population-wise oracle decomposition outside the certified
convex regime.

\subsection{UAV Crossing with Convex Scalar Kernels}
\label{scalar}

Consider two populations with single-integrator dynamics
\[
\dot{x}^a(t)=u^a(t),\qquad
x^a(t),u^a(t)\in\mathbb{R}^2,\qquad T=4,
\]
and heterogeneous control bounds
$\|u^1(t)\|\le 6$ and $\|u^2(t)\|\le 4.$
Population~1 moves from $x_0^1=(0,4)$ to $x_g^1=(8,4)$, whereas
Population~2 moves from $x_0^2=(4,0)$ to $x_g^2=(4,8)$.
Their nominal paths intersect at $(4,4)$, where a circular obstacle
with radius $r=0.6$ is placed. Both initial distributions are point masses, $\rho_0^a=\delta_{x_0^a}$; the spread visible in the figures arises from the mixture of trajectory atoms generated by the fully-corrective scheme.

The running and terminal costs are
$\ell_a(t,x,u)
=
\alpha\|u\|^2
+
\beta\max(0,r+\delta-\|x-c\|)^2,
\Psi_a(x)
=
\lambda_\Psi\|x-x_g^a\|^2,
$
with
$
\alpha=0.1$,
$\lambda_\Psi=20$,
$\delta=0.2$, and
$\beta=5000$.
All pairwise interaction kernels are the same Gaussian,
\[
W_{pq}(z)
=
\exp\!\left(-\frac{\|z\|^2}{2\sigma^2}\right),
\qquad
\sigma=1.
\]

The horizon is discretized using $N=150$ uniform time steps, and
$K=40$ fully-corrective FW iterations are performed. Each
population-wise subproblem~\eqref{eq:hetero-parametric-OCP} is solved on
this grid by gradient descent on the discretized control sequence, with
gradients obtained by the adjoint method and updates performed by Adam.
At each fully-corrective step, the convex-combination weights over
previously generated trajectory atoms are re-optimized through a
quadratic program over the simplex.

\paragraph{Scenario 1: Symmetric Convex Coupling.}

We set
\[
(\kappa_{pq})
=
\begin{pmatrix}
1.0 & 0.5\\
0.5 & 1.0
\end{pmatrix}.
\]
The Gaussian kernel is positive definite and
$\kappa_{11}\kappa_{22}=1.0\ge0.25
=\bigl(\tfrac{\kappa_{12}+\kappa_{21}}{2}\bigr)^{2}$, so
\eqref{eq:weight-test} holds and Theorem~\ref{thm:hetero-conv} applies. Figure~\ref{fig:slices_sym} shows that
the two populations split around the obstacle in a nearly symmetric
fashion during the crossing phase and subsequently reconverge toward
their respective targets.

\begin{figure}[t]
    \centering
    \includegraphics[width=\linewidth]{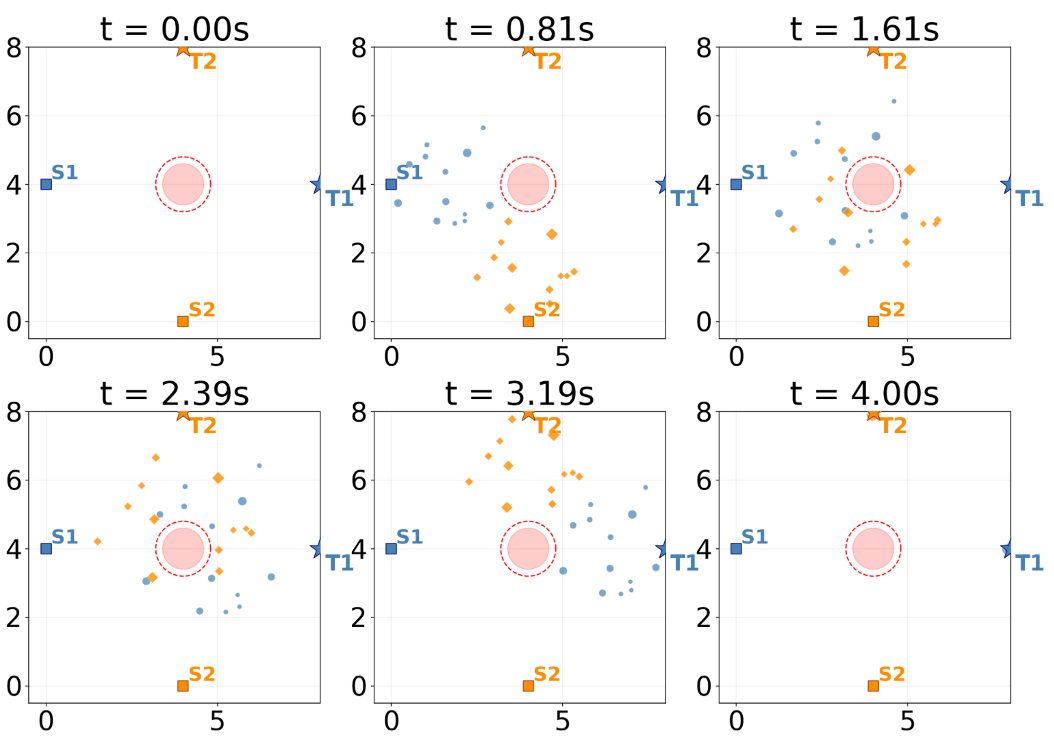}
    \caption{Scenario 1: symmetric convex coupling.}
    \label{fig:slices_sym}
    \vspace{-20pt}
\end{figure}

\paragraph{Scenario 2: Heterogeneous Convex Weighting.}

We next set
\[
(\kappa_{pq})
=
\begin{pmatrix}
1.0 & 0.8\\
0.1 & 0.3
\end{pmatrix}.
\]
Here $\kappa_{11}\kappa_{22}=0.30\ge0.2025
=\bigl(\tfrac{\kappa_{12}+\kappa_{21}}{2}\bigr)^{2}$, so
\eqref{eq:weight-test} again holds, with a smaller margin. Since $\mathbf K$ depends on $\kappa_{12}$ and
$\kappa_{21}$ only through their sum, this scenario
differs from Scenario~1 through the self-weights and
the effective symmetric cross-weight rather than
through any directed coupling. Figure~\ref{fig:slices_asym} shows that Population~1 develops a broader distribution and follows a somewhat longer route through the crossing region, whereas Population~2 remains more concentrated.
\begin{figure}[t]
    \centering
    \includegraphics[width=\linewidth]{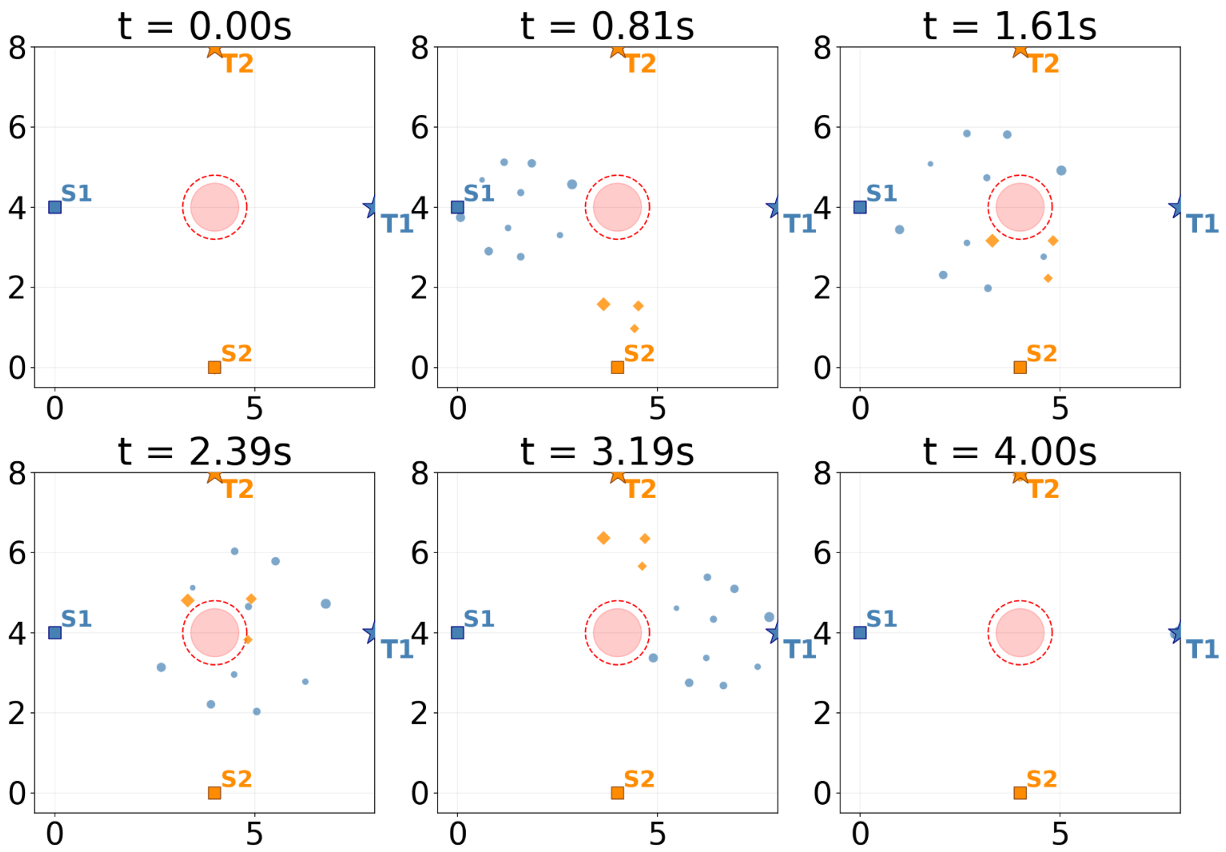}
    \caption{Scenario 2: heterogeneous convex weighting.}
    \label{fig:slices_asym}
\end{figure}
Figure~\ref{fig:obj_convex} compares the objective histories for the
two convex scenarios. In both cases, the objective decreases rapidly
during the early FW iterations and subsequently stabilizes as the
fully-corrective atom set is enriched.
\begin{figure}[t]
    \centering
    \vspace{-10pt}
    \includegraphics[width=1\linewidth]{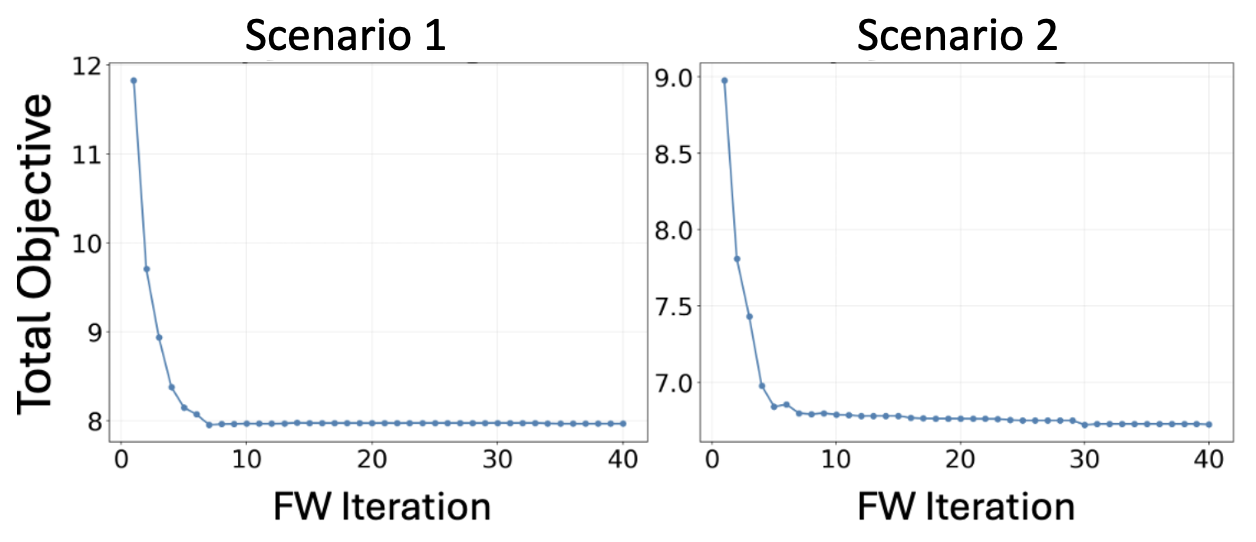}
    \caption{Objective histories for the two convex UAV crossing scenarios.}
    \label{fig:obj_convex}
    \vspace{-20pt}
\end{figure}
In both scenarios, the computed trajectories remain outside the
obstacle boundary. Since obstacle avoidance is imposed through a soft
penalty rather than a hard state constraint, this is an empirical
observation rather than an analytical safety guarantee.

\subsection{Directional 3D Search-and-Rescue}

We next consider a three-dimensional search-and-rescue scenario in
which Population~1 represents Rescue agents and Population~2
represents Search agents. Both populations satisfy
\[
\dot{x}^a(t)=u^a(t),\qquad
x^a(t),u^a(t)\in\mathbb{R}^3,\qquad T=5,
\]
with heterogeneous control bounds
$\|u^1(t)\|\le2.5$ and
$\|u^2(t)\|\le5$.
Both populations start from point masses at $(0,4,4)$ and move toward
$(12,4,4)$ while navigating four spherical obstacles.
Unlike the preceding examples, directional behavior is introduced
through the kernels themselves. To favor configurations in which
Search remains ahead of Rescue along the mission direction
$d=(1,0,0)$, we use
\begin{align}
W_{12}(z)
&=
\phi(z)
\left(
1+\varepsilon\tanh(\beta_d d^\top z)
\right),
\\
W_{21}(z)
&=
\phi(z)
\left(
1-\varepsilon\tanh(\beta_d d^\top z)
\right),
\end{align}
where
$
\phi(z)
=
\exp\!\left(-\frac{\|z\|^2}{2\sigma^2}\right)$,
$\sigma=1$,
$\varepsilon=0.6$,and $\beta_d=1.5.$
For the $W_{12}$ term,
$z=x_{\mathrm{rescue}}-x_{\mathrm{search}}$.
For $d^\top z>0$---that is, with Rescue ahead of Search---we have
$W_{12}(z)>W_{12}(-z)$, so that configuration is penalized more heavily
than its mirror image, and the ordering with Search ahead is
correspondingly favored.
Moreover,
$W_{21}(z)=W_{12}(-z)$,
so the two directional kernels consistently favor the desired
population ordering.

The remaining kernels are $W_{11}=W_{22}=\phi$, and the interaction weights are
$\kappa_{11}=1.0$, 
$\kappa_{22}=0.8$, 
$\kappa_{12}=0.9$, 
$\kappa_{21}=0.9$.
The cross kernels are not of the common form assumed
in~\eqref{eq:weight-test}: the factor
$1\pm\varepsilon\tanh(\beta_d d^\top z)$ is not even, which is precisely
what makes the interaction directional. Convexity is therefore not
certified for this example, which illustrates the population-wise
decomposition of Theorem~\ref{thm:hetero-lmo} outside the regime of
Theorem~\ref{thm:hetero-conv}.

Figure~\ref{fig:slices_dir} shows representative snapshots
of the computed population distributions. Search remains
ahead of Rescue along the mission direction while both
populations maneuver around the obstacles and approach
the target region. The instantaneous worst-case ordering
margin
\[
m(t):=
\min_j d^\top x_j^{\mathrm{search}}(t)
-\max_k d^\top x_k^{\mathrm{rescue}}(t)
\]
is nonnegative at all sampled times, with a sample
average of $0.373$. Rescue incurs control energy $41.4$,
compared with $31.1$ for Search, and the observed minimum
obstacle clearances range from $+0.20$ to $+0.64$. These quantities characterize the reported numerical solution only.
The directional interactions and obstacle penalties encourage the
observed ordering and collision avoidance but do not impose either as
a hard constraint.

\begin{figure}[t]
    \centering
    \includegraphics[width=\linewidth]{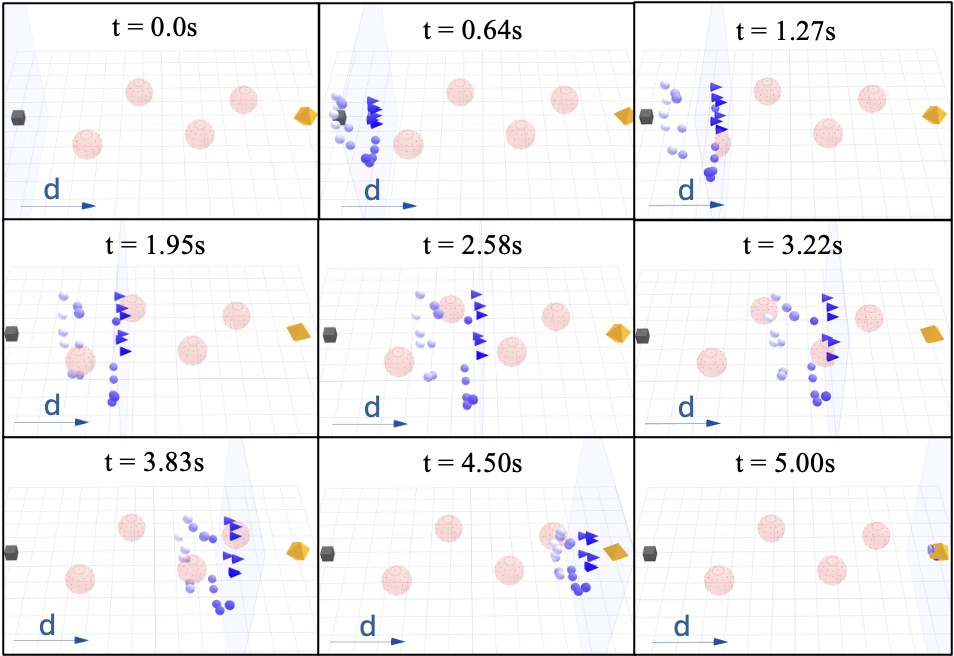}
    \caption{3D search-and-rescue: Rescue (purple),
Search (blue), obstacles (red).}
    \label{fig:slices_dir}
    \vspace{-20pt}
\end{figure}
\section{Conclusion}
\label{sec:conclusion}

We developed a heterogeneous occupation-measure formulation for
multi-population mean-field control. The main structural results are
threefold. First, positive semidefiniteness of the symmetrized
matrix-valued kernel ensures convexity and yields an
explicit weight test for a common positive definite kernel. Second, despite cross-population coupling in the original objective, the Frank--Wolfe linear minimization oracle remains population-wise separable after
linearization and admits classical optimal control realization under
additional assumptions. Third, the curvature of the heterogeneous
objective is bounded explicitly in terms of the horizon and the
symmetrized kernels, yielding the standard $O(1/k)$ rate for exact-oracle FW under the stated assumptions. The numerical examples illustrate this decomposition under heterogeneous
dynamics, interaction weights, and directional kernels. The present work
focuses on the population-level relaxed formulation; establishing
finite-agent approximation guarantees, general relaxation-gap analysis,
and convergence guarantees for nonconvex directional interactions remains
for future work.

\bibliographystyle{ieeetr}
\bibliography{ref}

\end{document}